\documentclass{article}
\usepackage{amsfonts,amsmath,amssymb,amsthm}
\usepackage[all,2cell]{xy}\UseAllTwocells\SilentMatrices

\usepackage{color}
\usepackage[utf8]{inputenc}
\usepackage{titlesec}
\usepackage{hyperref}
\hypersetup{
    colorlinks,
    citecolor=black,
    filecolor=black,
    linkcolor=black,
    urlcolor=black
} 

\usepackage[margin=1.5in]{geometry}
\usepackage{setspace}
\usepackage{comment}
\usepackage{tabularx}
\usepackage{comment}

\label{guassthm} 

\title{A Parametrization of D equivalences of coherent sheaves of symplectic resolutions of given symplectic singularity}
\author{D. Boger}
\date{}

	\newcommand{\bs}{\bigskip}

	\definecolor{red}{rgb}{1,0,0}

	\newcommand{\subsubsubsection}[1]{\vspace{0.5cm}
 \textit{\textbf{#1}} \vspace{0.2cm}}

	\newcommand{\g}{\mathfrak{g}}

	\newcommand{\Z}{\mathbb{Z}}
	\newcommand{\D}{\mathcal{D}}
	\newcommand{\F}{\mathcal{F}}
	\newcommand{\A}{\mathcal{A}}

	\newcommand{\R}{\mathbb{R}}
	\newcommand{\CC}{\mathbb{C}}
	\newcommand{\C}{\mathcal{C}}

	\newtheorem{example}{Example}
	\newtheorem{defi}{Definition}
	\newtheorem{claim}{Claim}
	\newtheorem{thm}{Theorem}
	\newtheorem{prop}{Proposition}
	\newtheorem{remark}{Remark}
	\newtheorem{lemma}{Lemma}
\newtheorem{observation}{Observation}

\begin{document}
\onehalfspacing
\maketitle
			
\newcommand{\Dl}{\mathcal{D_{\lambda}}}
\newcommand{\Aop}{\A_F^-}
\newcommand{\Dle}{\mathcal{D}_{\lambda,e}}
\newcommand{\ul}{u(\g)_{\lambda}}
\newcommand{\Al}{\Gamma(\Dl(G/P))}

\begin{abstract}

Let Y be a symplectic singularity over an algebraically closed field k. Let $X^{(i)}$ be its different symplectic resolutions. There is a question of understanding natural equivalence functors between the bounded derived categories of coherent sheaves $D^b(Coh(X^{(i)})), D^b(Coh(X^{(j)}))$. 
Let G be a reductive group over k. Consider associated parabolic subgroups $P^{(i)}$. Consider the varieties $X^{(i)}:=T^*(G/P^{(i)})$. By refining the construction from \cite{B1}, we construct a local system of categories that captures the natural equivalences between the categories. To each symplectic resolution $X^{(i)}$, there is a point $pt_{X^{(i)}}$ in $V^0_{\CC}$, for which the local system attaches the category $D^b(Coh(X^{(i)}))$. To each path between $pt_{X^{(i)}}$ and $pt_{X^{(j)}}$ the local system attaches a natural functor $D^b(Coh(X^{(i)}))\rightarrow D^b(Coh(X^{(j)}))$. This is a picture where the homotopy classes of paths between the points $pt_{X^{(i)}}, pt_{X^{(j)}}$ in $V^0_{\CC}$ parametrize natural equivalences between the corresponding bounded derived categories of coherent sheaves. At first we work with k of characteristic $p>>0$, and construct the local system for the categories $D^b(Coh_0(X^{(i)}))$. Then we extend it to the categories $D^b(Coh(X^{(i)}))$ and lift to characteristic zero.  

\end{abstract}

\tableofcontents
\bs

\bs

I'm very thankful to Prof. Bezrukavnikov, for his great help and support.

\section{Introduction}
\subsection*{The D equivalence question}
Let F be an algebraically closed field. 
Let $X^{(1)},X^{(2)}$ be two smooth projective varieties over F.

\begin{defi}
$X^{(1)}, X^{(2)}$ are \emph{K equivalent} if there is a smooth projective variety, Z, and birational correspondence $X^{(1)}\leftarrow Z\rightarrow X^{(2)}$, such that the pullbacks of the canonical divisors to Z, are linear equivalent.
\end{defi}

 A conjecture by Kawamata \cite{Ka} suggests that K equivalence of the varieties, implies an equivalence of the bounded derived categories of coherent sheaves, as triangulated categories. It's not expected that there will be a construction of a canonical equivalence. When the conjecture holds, one can further study the family of natural equivalences between the categories.

One case where the conjecture holds is the following. Let Y a symplectic singularity over F, Let $X^{(i)}$ be its different symplectic resolutions. These varieties are K equivalent, (the canonical divisor of a symplectic resolution is trivial). Kaledin proved \cite{Kal1} \cite{BK}, that the categories $D^b(Coh(X^{(i)}))$ are equivalent. His construction of the equivalence was non canonical. It involved a choice of a tilting generator for the category, which is a very non canonical object. We are interested in studying the family of natural equivalences between $D^b(Coh(X^{(i)})), D^b(Coh(X^{(j)}))$.

We specialize to the case that was described in the abstract and suggest a local system that captures the natural equivalences, as described in the abstract.

\section{Background}
				This section sets notations that will be used. It's a short version of the background section from \cite{B1}.
				
				\bs
				Let k be an algebraically closed field of characteristic $p>>0$. Let G be a reductive group over k. Let $\g$ be its Lie algebra. Let $\mathcal{B}$ be the variety of Borel subalgebras in $\g$. Let $\mathcal{N}\subset \g$ be the nilpotent cone. Let $T^*\mathcal{B}\rightarrow \mathcal{N}$ be the springer resolution.  
				
				Let T be a maximal torus. Let $\Lambda:=Hom(T,\mathbb{G}_m)$. Let $P$ be a parabolic subgroup s.t $T\subset P\subset G$. Let L be its Levi subgroup. Let $\Lambda_L:=Hom(L,\mathbb{G}_m)\subset \Lambda$. Let W be the weyl group. Let $T^*G/P\rightarrow \mathcal{N}_L$ be the parabolic springer map. $\mathcal{N}_L\subset \mathcal{N}$.
				
				\begin{defi} $\lambda\in \Lambda$ is called p-regular, if the stabilizer in W of $\lambda+p\Lambda\in \Lambda/p\Lambda$ is trivial. 
				\end{defi}

			\subsection*{D modules and Coherent sheaves in characteristic p}
			
					Let k be a perfect field of characteristic $p>>0$. Let X be a smooth variety over k. 
					
					Let $\mathcal{D}_X$ be the sheaf of crystalline differential operators on X. 
					\bs
					
					Twisted differential operators. Let $\mathcal{L}\in Pic(X)$ be a line bundle. Then the sheaf of $\mathcal{L}$ twisted differential operators  $\mathcal{D}_X^{\mathcal{L}}$ is well defined. 
					\bs
					
					Notation: Let $X:=G/B$ let $\mathcal{D}_{\lambda}:=\mathcal{D}^{\mathcal{L}}$, where $\mathcal{L}$ corresponds to $\lambda$ under the equivalence $Pic(G/B)\simeq \Lambda$. Similarly for $X:=G/P$ where $Pic(G/P)\simeq \Lambda_L$
				\bs
				
				Notation: Let $\mathcal{D}_X^{\mathcal{L}}-mod$ be the category of $\mathcal{L}$ twisted D modules on X. (For X:=G/P we also denote it $\mathcal{D}_{X,\lambda}-mod$ or $\mathcal{D}_{\lambda}-mod$ when there is no ambiguity)
					\begin{claim}
					It is still true, as in characteristic 0, that $\mathcal{D}_X$ is a quantization of $O_{T^*X}$. That is, there is a natural filtration on this sheaf and $gr(\mathcal{D}_X)\simeq O_{T^*X}$. 
					\end{claim}
					
					\begin{thm}
					The center of $\mathcal{D}_X$ is big. It's canonically isomorphic to the sheaf of rings $O_{(T^*X)^{(1)}}$. Here the superscript (1) stands for Frobenius twist. (Remark: For a variety over a perfect field k, of char p, which is defined over $F_p$, the variety and its Frobenius twist are isomorphic as k schemes, hence we usually omit the superscript from the notation and identify $T^*X^{(1)}\simeq T^*X$)
					\end{thm}
					
					\begin{thm}
							$\mathcal{D}_X$, and $\mathcal{D}_X^{\mathcal{L}}$ are both Azumaya algebras on $T^*X^{(1)}$.
					\end{thm}
					\begin{claim}
						$\mathcal{D}_X^{\mathcal{L}}$ and $\mathcal{D}_X$ are Morita equivalent. There is a canonical equivalence between the category of $\mathcal{D}_X$ modules and the category of $\mathcal{D}_X^{\mathcal{L}}$ modules. 
					\end{claim}
						
					\begin{thm}	
						$\mathcal{A}:=\mathcal{D}_X$ doesn't split over $T^*X$, but it does splits on a formal neighborhood of the zero section of $T^*X$. 
										
					The splitting implies the equivalence:
					\begin{equation}
					\label{D-coh}
					\mathcal{D}_X^{\mathcal{L}}-mod_0\simeq Coh_0(T^*X)
					\end{equation}
					Where in both side we restrict to sheaves with support on the formal neighborhood of the zero section $X\subset T^*X$.
					\end{thm}
					\begin{defi}
					\label{F D-coh}
					Let $F_{D^{\mathcal{L}},Coh_{T^*X}}$ denote an equivalence functor from the left hand side to the right hand side. In the case that $X:=G/P$ we also denote it by $F_{D_{\lambda},Coh_P}$ or $F_{D,Coh_P}$ when there is no twist.
					\end{defi}

			\subsection*{Localization theorem in characteristic p}
			
			Let $\mathfrak{h}$ be the universal Cartan subalgebra of $\g$. Let $B\subset G$ be any Borel subgroup of G, and $\mathfrak{b}$ be its lie algbera, then there is a canonical isomorphism $\mathfrak{h}\simeq \mathfrak{b}/[\mathfrak{b},\mathfrak{b}]$. Let $T\subset B$ be a maximal torus contained in B. Let $\mathfrak{t}$ be its lie algebra.
			
			Let W be the Weyl group. W acts on the weight lattice $\Lambda$ by the dot action. $w\bullet\lambda:=w(\lambda+\rho)-\rho$.  
						
			Define $\xi$ to be the center of $U(\g)$. The center consists of two parts. The Harish Chandra center $\xi_{HC}:=U(\g)^G\simeq S(\mathfrak{t})^{(W,\bullet)}$, $S(\mathfrak{t})$ is the symmetric algebra of $\mathfrak{t}$. And The Frobenius center (p-center) $\xi_{Fr}\simeq S(\g^{(1)})$. To give a central character of $U(\g)$, we need to give a pair of elements $(\lambda,e)\in \mathfrak{t^*}\times \g^{*(1)}$, which are compatible. In particular lets restrict to the case that e is nilpotent and $\lambda$ an integral weight.

				\bs
				
				Let $\lambda\in \Lambda$. There is a natural map $\Lambda/p\Lambda\rightarrow \mathfrak{h}^*/W$ $\lambda\mapsto d\lambda$. Hence $\lambda$ defines a maximal ideal of $\xi_{HC}$. We define $U(\g)_{\lambda}:=U\otimes_{\xi_{HC}}k$. 
				
				Let $U(\g)_{\lambda}-mod$ be the category of finitely generated $U(\g)_{\lambda}$ modules.

				Given a compatible pair $\lambda\in \Lambda,e\in \g^*$. We consider $U_{\lambda}^{\hat{e}}$-mod. The category of finitely generated $U^{\lambda}$ modules with generalized p character $\lambda$.  We also denote this category by $U_{\lambda}-mod_{e}$.

				\subsubsection*{Derived localization in char p}
						See \cite{BeiBer} for the classical BB localization theorem. The following theorem, proved in \cite{BMR},\cite{BMR2} is an analog in characteristic p.
						
						Let $\lambda\in \Lambda$.
						
						\begin{thm}
						There is a natural isomorphism $\Gamma(\mathcal{D}_{{G/B},\lambda})\simeq U(\g)_{\lambda}$. 
						\end{thm}
														
						\begin{thm}
						Let $\lambda\in \Lambda$ be p-regular, then the derived global sections functor $\Gamma^{B}_{\lambda}:D^b(\mathcal{D}_{G/B,\lambda}-mod)\rightarrow D^b(U_{\lambda}-mod)$ is an equivalence of categories. \cite{BMR}
						\end{thm}
						\begin{defi}
								Let $Loc^B_{\lambda}:D^b(U_{\lambda}-mod)\rightarrow D^b(\mathcal{D}_{G/B,\lambda}-mod)$ denote the derived localization functor.
						\end{defi}
						
						\begin{thm}
						Let $\lambda,e$ be a compatible pair. The derived global sections induces an equivalences on the subcategories $\Gamma^{B}_{\lambda,e}:D^b(\mathcal{D}_{\lambda}-mod_e)\rightarrow  D^b(u_{\lambda}-mod_e)$
						\end{thm}
				
						Note that combining the last theorem with the equivalence in equation (\ref{D-coh}), the following equivalence is obtained: $D^b(Coh_0(T^*G/B))\simeq D^b(\mathcal{D}(G/B)_{\lambda}-mod_0)\simeq D^b(u(\g)_{\lambda}^{\hat{0}}-mod)$. 
							\bs
							
		Similarly, there is a parabolic version
			\begin{defi}
			
			Let $X:=G/P$, P a parabolic. Let $A_{\lambda,(G/P)}:=\Gamma(G/P,\mathcal{D}_{G/P,\lambda})$. Let $A_{\lambda,(G/P)}-mod$ be the category of finitely generated modules over $A_{\lambda,(G/P)}$. Let $A_{\lambda,G/P}^{\hat{0}}-mod\subset A_{\lambda,(G/P)}-mod$ be the subcategory of finitely generated $A_{\lambda,(G/P)}$ modules with generalized p center, zero. We also denote it $A_{\lambda,(G/P)}-mod_0$.
			\end{defi}
			
			\begin{claim}
						Let P,Q be associated Parabolic subgroups. There is a canonical equivalence  $Pic(G/P)\simeq Pic(G/Q)$ Let $\lambda\in Pic(G/P)\simeq Pic(G/Q)$, then the algebras $A_{\lambda,(G/P)}, A_{\lambda,G/Q}$ are canonically equivalent. (See \cite{B1} for proof).
			\end{claim}
			This justifies omitting P from the notation for $A_{\lambda}$
			\begin{defi}
					Let $A_{\lambda}:=A_{\lambda,(G/P)}$, Let $A_{\lambda}-mod$ be the category of finitely generated $A_{\lambda}$ modules. 
										Let $A_{\lambda}^{\hat{(0)}}-mod$ be the subcategory of finitely generated $A_{\lambda}$ modules, with generalized p center zero. 
										We also denote it $A_{\lambda}-mod _0$
			\end{defi}
			
			\begin{thm} Let $\lambda\in \Lambda_L$ be p regular. then the derived global sections functor $\Gamma^P_{\lambda}:D^b(\mathcal{D}_{G/P,\lambda}-mod)\rightarrow D^b(A_{\lambda}-mod)$ is an equivalence of categories,and  $D^b(Coh_0(T^*G/P))\simeq D^b(\mathcal{D}(G/P)_{\lambda}-mod_0)\simeq D^b(A_{\lambda}^{\hat{0}}-mod)$.  \cite{BMR2}. 
			\end{thm}
			
			\begin{defi}
			Let $\Gamma_{\lambda}^P$ be the derived global sections functor.
			
					Let $Loc^P_{\lambda}$ denote the derived localization functor. 
					
					Let $F_{D_{\lambda},Coh}^P$ be an equivalence functor $D^b(\mathcal{D}_{\lambda}-mod_0)\rightarrow D^b(Coh_0(T^*G/P))$
					
					Let $F_{A_{\lambda},Coh}^P$ be an equivalence functor $D^b(A_{\lambda}-mod_0)\rightarrow D^b(Coh_0(T^*G/P))$
					
			\end{defi}

				\subsection*{A Partial order imposed by a cone}
				\begin{defi}
				Let $V_{\R}$ be a real vector space. Let $V_{\R}^0$ be the complement of a real hyperplane arrangement in $V_{\R}$. We call the connected components, alcoves. Let $C\subset V_{\R}$ be a cone. Define: two alcoves $\A,\A'$ have the relation \emph{$\A'$ is above $\A$ with respect to C}, if $\A'\subset \A+C$. Notation $\A'>_{C} \A$
			
				\end{defi}
				\begin{remark}
				Given two alcoves in $V^0_{\R}$ that share a codimention 1 wall. Let $V^0_{\CC}$ be the complexification of $V^0_{\R}$. Then, $\A'>_{C}\A$ iff the positive half loop from $\A$ to $\A'$ belongs to $V^0_{\R}+iC$. 
				\end{remark}
				
				\begin{defi}
				In particular, let $V_{\R}:=\Lambda\otimes \R$. Let $C_P\subset \Lambda\otimes \R$ be the cone of positive weights for the parabolic P. We will soon define $V^0_{\R}$ to be a complement of a hyperplane arrangement in $\Lambda\otimes \R$. We get the partial order on the alcoves, $\A<_{C_P} \A'$. We denote it also by $\A<_P \A'$ . Moreover, let $\lambda,\mu\in V_{\R}^0$ be points in $\A,\A'$ respectively, then we also denote $\lambda<_P\mu$.

				\end{defi}

\bs

As described in the abstract, we answer the D equivalence question by a construction of a local system of categories over a topological space $V^0_{\CC}$. 

\section{The construction}
\subsection*{The topological space $V^0_{\CC}$}

Let $X\rightarrow Y$ be a symplectic resolution over k. We attach the following topological space. 
\begin{defi}
Let $V_{\R}:=H^2(X,\R)$. (Observe that for a symplectic resolution X, there is a canonical equivalence $H^2(X,\R)\simeq Pic(X)\otimes \R$). This is the universal parameter space for quantizations of X, hence we can ask whether derived localization holds for the quantization associated to the point $\lambda\in V_{\R}$. This real vector space is canonically identified for different symplectic resolutions of Y. Let $V_{\R}^0\subset V_{\R}$ be the area where derived localization theorem does not hold. It's a conjecture, that for every symplectic resolution, this is a complement of an hyperplane arrangement. For the case $X=T^*G/P$ this is known. Let $V_{\CC}:=V_{\R}^0\otimes \CC$. 
\end{defi}

$V_{\CC}$ is connected.

\begin{example}
Let $X:=T^*G/P$, $V_{\R}^0\subset \Lambda_L\otimes \R$ is the p-regular weights.
\end{example}

\begin{defi}
We call the connected components in $V_{\R}^0$, \emph{the alcoves}.  
Let $\lambda\in V_{\R}^0$. We let $\mathcal{A}_{\lambda}$ be the alcove that contains $\lambda$.
\end{defi}

\subsection*{The local system}
Let $P_0\subset G$ be a fixed parabolic subgroup. Let L be its Levi subgroup.
Let Y be $\mathcal{N}_L$. Let $X^{(i)}$ be $T^*G/P^{(i)}$ for different Parabolic subgroups $P^{(i)}$ that are associated to $P_0$. The categories $D^b(Coh_0(X^{(i)}))$ are all equivalent. We construct a local system of categories on $V^0_{\CC}$, with values the categories $D^b(Coh_0 X^{(i)})$ and $D^b(A_{\lambda}-mod_0)$, $\lambda$ a p-regular weight in $\Lambda_L$. These categories are equivalent, however we choose to assign different realizations of the categories to different points, so that it allows us to describe the functors attached to paths between points in a natural way. In particular,  for every $X^{(i)}$, there is a special point $pt_{X^{(i)}}\in V_0^{\CC}$ (A contractible subset of points), to which the local system assigns the category $D^b(Coh(X^{(i)}))$. This suggests a picture, where natural D equivalences between the derived categories $D^b(Coh X^{(i)}), D^b(Coh X^{(j)})$, are parametrized by homotopy classes of paths between the associated points $pt_{X^{(i)}}, pt_{X^{(j)}}$ in the topological space $V^0_{\CC}$. For abbreviation, we denote the point $pt_{T^*G/P^{(i)}}$ by $pt_{P^{(i)}}$.

In the more general setup of any symplectic resolution, Namikawa constructed an action \cite{Na} of a generalized Weyl group on the parameter space $V^0_{\CC}$. It may be possible to generalize the picture suggested here to other symplectic singularities, using Namikawa's action. \cite{Na} 
\bs

Recall from \cite{B1}, to construct a local system of categories on a complexification of a complement of hyperplanes, $V^0_{\CC}$, we can use Salvetti's generators and relations of the groupoid of $V^0_{\CC}$. In order to see the points $pt_{P^{(i)}}$ in the picture, we use a refinement of these generators and relations.

\subsection*{Refinement of Salvetti's generators and relations.}
We want to construct a functor from the groupoid on $V^0_{\CC}$ to Cat. (Attaching categories to points in $V^0_{\CC}$, attaching functors up to isomorphism to homotopy classes of paths between points, s.t concatenation of paths corresponds to composition of the functors. Not discussing higher compatibilities).
$V^0_{\CC}$ is the complexification of the complement of a collection of hyperplanes. Recall from \cite{B1} Salvetti generators and relations of the groupoid. 
The points to which we attach a category, are one point for each real alcove. The paths for which we attach a functor between the categories (the \emph{generators} of the groupoid) are the following, 
Let $\mathcal{A}, \mathcal{A'}$ be two real alcoves in $V^0_{\R}$ that share a codimension one wall. Let $l_{\mathcal{A},\mathcal{A'}}$ be the half loop in the positive direction from $\mathcal{A}$ to $\mathcal{A}'$.  
$l_{\mathcal{A},\mathcal{A'}}$ are the generators.

To refine it, for each parabolic P, consider the positive cone of weights $C_P\subset V_{\R}$. Consider the subset, $V^0_{\R}+iC_{P}\subset V^0_{\C}$, Let the added points $pt_P$ be any point in this set. 

Given a parabolic P, there is a natural path (line) from the point $pt_P$ to a point in each of the real alcoves. The refined generators are taken to be these paths between any real alcove to the points $pt_{P^{(i)}}$.

\subsection*{The definition of the local system}

We attach the following categories to the points: To a point in the real alcove $\mathcal{A}_{\lambda}$ attach $D^b(A_{\lambda}-mod_0)$. To a point $pt_P$ we already said we attach $D^b(Coh_0(T^*G/P))$.

For the line from an alcove $\mathcal{A}_{\lambda}$ to the point $pt_P$ attach the functor 
\begin{equation}
F_{A_{\lambda}, Coh_P}:D^b(A_{\lambda}-mod_0)\rightarrow D^b(Coh_0(T^*G/P))
\end{equation} 

\begin{equation}
F_{A_{\lambda},Coh_P}:=F_{D,Coh_P}(\otimes O(0-\lambda)) Loc_{\lambda}^P
\end{equation}

$O(0-\lambda)\in Pic(T^*G/P)\simeq \Lambda_L$, $F_{D,Coh}$ was defined above in definition \ref{F D-coh}.

\bs

Observe, given an old generator, $l_{\A_{\lambda},\A'_{\mu}}$ half a loop in the positive direction from the alcove $\mathcal{A}_{\lambda}$ to $\mathcal{A}_{\mu}$. Let P be a parabolic associated to $P_0$. The half loop is contained in $V^0_{\R}+iC_P$ iff 

\begin{equation}
\label{order}
\lambda<_P\mu
\end{equation}

Hence this old generator breaks to the natural path from $\mathcal{A}_{\lambda}$ to $pt_P$ and a path from $pt_P$ to $\mathcal{A}_{\mu}$, for each parabolic P, associated to $P_0$,  whose induced order satisfies equation (\ref{order}). 

Observe that for this old generator, the functor that we attach here  between $D^b(A_{\lambda}-mod_0)\rightarrow D^b(A_{\mu}-mod_0)$ is $\Gamma^P_{\mu}\otimes O(\mu-\lambda) Loc^P_{\lambda}$ for a parabolic P(associated to $P_0$) such that $\lambda<_P\mu$. Denote this functor $F_{\lambda,\mu}$. It follows from \cite{B1} that this is independent of the choice of P and that this is a well defined local system on $V^0_{\CC}$ with values in Cat.		
\bs

\section{Lifting to characteristic 0}		
\section*{Lifting the local system to char 0, and to case without a support condition.}		

$P_0$ is our fixed parabolic. Lets discuss the local system using the non-refined Salvetti generators $l_{\A,\A'}$. 
\begin{defi}
Using the canonical equivalence $F_{u_{\lambda},Coh_{P_0}}:D^b(A_{\lambda}-mod_0)\rightarrow D^b(Coh_0(T^*G/P_0))$

\begin{equation}F_{u_{\lambda},Coh_{P_0}}:=F_{D,Coh_{P_0}} \otimes O(0-\lambda)Loc_{P_0,\lambda}
\end{equation}
we consider the above as a local system with constant value the category $D^bCoh_0(T^*G/P_0)$.
\end{defi}

In this setup, the functors attached to half loops in positive directions between adjacent alcoves that share a codim 1 wall are:
\begin{defi}
$F_{\lambda,\mu,Coh_{P_0}}:=F_{u_{\lambda},Coh_{P_0}}^{-1}F_{\lambda,\mu}F_{u_{\mu},Coh_{P_0}}$
\end{defi}

If $\lambda<_{P_0}\mu$, then $F_{\lambda,\mu,Coh_{P_0}}\simeq Id$

\bs

The above local system was constructed when G lives over k, the characteristic p field, and the value was the category of bounded derived coherent sheaves with support condition. We would like to extend. We would like to construct a local system on $V^0_{\CC}$ with value the category $D^b(Coh(T^*G/P_0))$,(without the support condition), s.t the restriction of the functors to the subcategory $D^b(Coh_0(T^*G/P_0))$ recovers the above local system. We would also like to extend to case that $G$ lives over a characteristic zero field/ring.

 We discuss the kernels by which these functors $F_{\lambda,\mu,Coh_{P_0}}$ are defined, and some properties of these kernels with respect to base change.

\bs

Let $R$ be a $\Z[1/h]$ algebra, where h is the Coxeter number. 
Let $G_{R}$ be a split connected, simply connected, semi simple algebraic group over $R$. Consider its corresponding lie algebra $\g_{R}$. Consider a maximal torus and a Borel subgroup that contains it and their lie algebras respectively. $T_{R}, B_{0,R}, \mathfrak{t}_{R}, \mathfrak{b}_{0,R}$. Let $\Lambda$ be the weight lattice. Consider the corresponding root system $\phi$, together with positive roots $\phi^+$ for $\mathfrak{b}_{0,R}$ and the simple roots $\Sigma$. Consider the Weyl group W. Consider the set I of Coxeter generators of W, $s_{\alpha}$ associated to $\Sigma$. Consider the affine Weyl group $W_{aff}:=W\ltimes \Lambda$. It's a Coxeter group. To each Coxeter group there is an associated braid group. Consider the associated Braid group $Br_{aff}$. 

Notation: We denote a base change to a ring $R'$ by subscript $_{R'}$.

\subsubsection*{$F_{\lambda,\mu,Coh_{B_0}}$ and the classical action of the affine Braid group on $D^b(Coh(T^*G/B_0))$}

There exists an action of $Br_{aff}$ on $D^b(Coh(T^*G/B_0)_R)$, constructed by Bezrukavnikov and Riche in \cite{BR}\cite{R}. By base changing to an algebraically closed field ,k, of characteristic $p>h$, and then restricting to a formal neighborhood of the zero section $G/B_0\subset T^*G/B_0$, this induces an action on $D^b(Coh_0(T^*(G/B_0)_k)$. Observe that $\pi_1(V^0_{\CC})\simeq Br_{aff,pure}\neq Br_{aff}$. However in the case $G/B_0$, there is an additional symmetry on the parameter space $V^0_{\CC}$. $W_{aff}$ acts on the alcoves in $V^0_{\R}$, simply transitively, and the affine braid group is the fundamental group for the quotient space $\pi_1(V^0_{\CC}/W_{aff})\simeq Br_{aff}$.

\begin{defi}
Consider the projection $Br_{aff}\rightarrow W_{aff}$. There exists a natural section of the projection $Br_{aff}\rightarrow W_{aff}$, [BMR2,2.1]. Denote this section $w\mapsto T_w$. It's not a group homomorphism, however, there is the length function $l:W_{aff}\rightarrow \mathbb{N}$, and the section satisfies the rule $T_{w_1w_2}=T_{w_1}T_{w_2}$ when $l(w_1)+l(w_2)=l_{w_1w_2}$.
\end{defi}

Let $F_w$ denote the action functor of $T_w$ on $D^b(Coh(T^*G/B_0))_{\Z[1/h]}$. More generally, let $F_{w,R}$ denote the action functor on a base change to R, $D^b(Coh(T^*G/B)_R)$. In particular Let $F_{w,k}$ denote the action of $T_w$ on $D^b(Coh_0(T^*G/B_{0})_k)$. Let $\lambda,\mu$ belong to two adjacent alcoves in $V^0_{\R}$, that share a codimension one wall. There exists a unique $w\in W_{aff}$ s.t $\mu=w.\lambda$. If $\lambda>_{B_0} \mu$, then $F_{w,k}\simeq F_{\lambda,\mu,Coh_{B_0}}$. 

The functors $F_{w,R}$, have kernels $\mathcal{F}_w\in D^b(Coh(T^*(G/B_0)_R\times_R T^*(G/B_0)_R))$, which are in fact in the heart of the natural t structure, and are simple to describe. We recall their description from the work of \cite{BR}. 

\bs

Recall,

\subsection*{A presentation of $Br_{aff}$}

There exist several natural presentations of $Br_{aff}$. We consider the following one:
\bs

Generators: Let $s\in I$ be the simple reflections. To each simple reflection we associate a generator of $Br_{aff}$ called $T_{s}$.  ($T_{s}$  will indeed be the natural lift of $s\in W_{aff}$ under the projection $Br_{aff}\rightarrow W_{aff}$ mentioned above). Another form of generators are $\theta_x$ $x\in \Lambda$. 

\bs

Relations: Let $s,t\in I$ be simple reflections. Let $n_{\alpha,\beta}$ denote the order, $Ord(s,t\in W)$ of $st\in W$. It's enough to impose the following relations: 

Relations between the $T_{s}$ $s\in I$ is the braid relations: $T_{s}T_{t}...=T_{t}T_{s}...$  $n_{s,t}$ compositions on each side.

Relations between the $\theta$, are saying that $\theta:\Lambda\rightarrow Br_{aff}$ is a group homomorphism, $\theta_x\theta_y=\theta_{x+y}$. 

Relations between $T_{s}$ and $\theta_x$: They commute if $\alpha(x)=x$, and otherwise, $\theta_x=T_{s}\theta_{x-\alpha(x)}T_{s}$, where $s:=s_{\alpha}$ the simple reflection for a simple root $\alpha$.

\subsection*{The kernels of the functors $F_{\lambda,\mu,Coh_{B_0}}$}

The kernel of $F_{\lambda,\mu,Coh_{B_0}}$ has a nice description. As proved in \cite{BR},\cite{R}, there is a nice description of the kernel of the action $F_{w,R}$. This description is compatible with base change. 

The description is most natural in the Koszul dual situation \cite{MR} where $T^*G/B_R\rightarrow \mathcal{N}_R$ is replaced by the Grothendieck resolution $\tilde{\g_R}\rightarrow \g_R$. (The fact that the Grothendieck resolution is small, whereas the springer resolution is only semi small, makes the Grothendieck resolution better to work with). 
\bs

\subsubsection*{The idea of the kernels construction}

Consider the case $w\in W\subset W_{aff}$.
Let $F_{w,R}^{\tilde{\g}}$ be the notation for the functor which is the action of $T_w$ on the category $D^b(Coh(\tilde{\g}_R))$.

Let $\pi:\tilde{\g}\rightarrow \g$. Let $\g_{reg}\subset \g$ the regular elements.  Let $\tilde{\g}_{reg}:=\pi^{-1}(\g_{reg})$.

The kernels that describe $F_w^{\tilde{\g}}$ are constructed by the following idea:
Recall the construction of the action of the Weyl group W on the cohomology of a springer fiber $\mathcal{B}_e$ $e\in \mathcal{N}$, preformed using the minimal Goresky MacPherson extension of a perverse sheaf \cite{L1}. The idea is to use the fact that $\tilde{\g}_{reg}$ is a ramified Galois covering with Galois group W. Thus W acts on $\tilde{\g}_{reg}$ by deck transformations. Then, this action extends to an action on the cohomology of the springer fiber. 

The idea of the construction of the Braid group action on $D^b(Coh(\tilde{\g}))$ is similar. Starting the action of W on $\tilde{\g}_{reg}$ and extending it to action on the category of derived coherent sheaves on $\tilde{\g}$.

\begin{remark} When restricting $\tilde{\g}\rightarrow \g$ to the regular semi simple locus this is a cover, hence W acts on $\tilde{\g}_{reg-ss}$ by Deck transformations. Then, this action can be extended to the regular locus.
\end{remark}

The formula of the functors of $F_{w}^{\tilde{\g}}$ in terms of kernels is as follows:

\subsubsection*{The kernels construction}

Let $(B_0)_R$ be the fixed Borel, let $\mathcal{B}_R$ be the flag variety. Let $w\in W\subset W_{aff}$. Let $T_w\in Br_{aff}$. The action of $T_w$ on $D^b(Coh(\tilde{\g}_R))$ is defined in \cite{BR} to be a functor $F_{w,R}^{\tilde{\g}}$, whose kernel is as follows. 

\begin{defi}
Let $\Gamma_{w,R}$ be the graph of the action of $w\in W$ on $\tilde{\g}_{reg,R}$. Let $Z_{w,R}$ be the closure of the graph. Another description is by considering the natural morphisms $\tilde{\g_R}\times_R \tilde{\g_R}\rightarrow \mathcal{B}_R\times_R \mathcal{B}_R$. $Z_{w,R}$ is the inverse image of the $G_R$ orbit, $(B_{0,R},w^{-1}B_{0,R})$.  
\end{defi}
\begin{defi}
Let the kernel of $F_{w,R}^{\tilde{\g}}$ be $O_{Z_{w,R}}\in D^b((T^*G/B_0\times T^*G/B_0)_R)$
\end{defi}

\begin{defi}
In the dual version where the category is $D^bCoh(T^*G/B_0)_R$, the kernel is given by the structure sheaf of a closed subscheme of $T^*G/B_{0,R}\times_R T^*G/B_{0,R}$. The subscheme is defined to be the scheme theoretic intersection $Z_{w,R}':=Z_{w,R}\cap (T^*G/B_{0,R}\times _R T^*G/B_{0,R})$. The functors in this case are denoted $F_{w,R}$.
\end{defi}

\begin{claim}\cite{BR}
These functors $F_{w,R}^{\tilde{\g}}$ , $F_{w,R}$ define an action of the Braid group on the categories $D^b(Coh(\tilde{\g}_R)), D^b(Coh(T^*G/B_0)_R)$ respectively. In other words, the functors satisfy the braid relations. 
\end{claim}

\subsubsection*{The affine braid group action in terms of kernels}
\begin{remark}
Above we recalled the action of $T_w$ for $w\in W$. The extra action of $\Lambda\subset Br_{aff}$, is given by twist by the line bundles corresponding to the lattice elements under the equivalence $\Lambda\simeq Pic(T^*G/B)$. Let $\theta_x$ $x\in \Lambda$ act by twist by the line bundle corresponding to x. The kernel is the direct image of $O_{(T^*G/B)_R}(\lambda)$ under the diagonal embedding, $\Delta\hookrightarrow (T^*G/B)_R\times (T^*G/B)_R$. It's clear that the relation $\theta_x\theta_y\simeq \theta_{x+y}$ for $x,y\in \Lambda$ holds. The relation between $T_{s},\theta_x$ holds as well. 
\end{remark}

\begin{observation}
	Let $R\rightarrow R'$ be a ring morphism. By definition, the kernel of the functor $F_w^{\tilde{\g}_{R'}}$ is a base change of the kernel of $F_w^{\tilde{\g}_{R}}$
	\bs
	
	A similar observation is true for the kernels of the actions of $T_w$ on $D^b(Coh(T^*G/B_R))$
\end{observation}

The actions of $Br_{aff}$ on $D^b(Coh(\tilde{\g})_R), D^b(Coh(T^*G/B)_R)$ are \emph{weak geometric actions}. See next subsection for the definition of that notion.

\bs

\subsubsection*{Aside: Convolution of kernels and weak geometric action}

Let X be a smooth scheme over a commutative ring R. 
Consider the two natural projections $\pi_1,\pi_2:X\times _R X\rightarrow X$. Consider a sheaf $\mathcal{F}\in D^bCoh(X\times_R X)$. We say that $\mathcal{F}$ is good, if it is in the image of the derived pushforward from derived coherent sheaves on a closed subscheme of $X\times_R X$,  for which the two projections to $X$ are proper. If $\mathcal{F}$ is a good sheaf, then it's a kernel of a well defined functor $F_{\mathcal{F}}:D^b(Coh(X))\rightarrow D^b(Coh(X))$ 

\begin{equation}F_{\mathcal{F}}(\mathcal{G}):=\pi_{2*}(\mathcal{F}\otimes \pi_1^* \mathcal{G})
\end{equation}
 (Where all operations on the Right hand side are derived)

The formula for the composition of two such functors in terms of the kernels is well known: 
 
\begin{claim}

Let $\mathcal{F}_1, \mathcal{F}_2$ be two kernels that defines functors $D^bCoh(X)\rightarrow D^bCoh(X)$.
The composition of $\mathcal{F}_1,\mathcal{F}_2$ is another functor given by a kernel. The kernel is the convolution of the two kernels. Let $\pi_{12},\pi_{13},\pi_{23}:X\times _R X\times_R X\rightarrow X\times _R X$ be the three possible projections. Then
\begin{equation}
\mathcal{F}_1\star \mathcal{F}_2:= \pi_{13*}(\pi_{12}^*\mathcal{F}_2\otimes \pi_{23}^*\mathcal{F}_2)
\end{equation}
 (Again all operations are derived)
\end{claim}

In other words - $D^b(X\times X)$ has a monoidal structure given by the convolution product, and this monoidal category acts on the category $D^b(Coh(X))$, letting $\mathcal{F}$ act by $F_{\mathcal{F}}$.

\begin{observation}
The above claim is compatible with base change to another base ring R'. That is, let $R\rightarrow R'$, let $X_{R'}$ be the base change. Consider the map
$B:X_{R'}\times_{R'} X_{R'}\rightarrow X\times _R X$. The derived pullback of the kernel $B^*:D^bCoh(X\times _R X)\rightarrow D^bCoh(X_{R'}\times _{R'} X_{R'})$ satisfies $B^* \mathcal{F}_1\star B^*\mathcal{F}_2\simeq B^*(\mathcal{F}_1\star \mathcal{F}_2)$. In other words, the pullback functor $B^*:D^bCoh(X_R\times _R X_R)\rightarrow D^bCoh(X_{R'}\times _{R'}X_{R'})$ is a monoidal functor. The action of $D^b(X_{R',R}\times_{R',R} X_{R',R})$ on $D^b(X_{R',R})$ is compatible with this pullback map.
\bs

Let Y be a variety over R, Let $X:=T^*Y$, then the formalism of convolution is also compatible with restriction to a formal neighborhood of the zero section $Y\subset T^*Y$. 
\end{observation}

These observations leads to the definition of a weak geometric action of a group on a category. As follows. \cite{BM}

\begin{defi}
A \emph{weak homomorphism} from an abstract group H to a monoidal category, is a morphism from H to the group of isomorphism classes of invertible objects. 
\end{defi}

\begin{defi}
A \emph{weak geometric action} of an abstract group H on the category $D^bCoh(X_R)$ is a weak morphism from H to $D^bCoh(X_R\times _R X_R)$. 
\end{defi}

Let $R\rightarrow R'$ be a commutative ring morphism. A weak geometric action on $D^bCoh(X_R)$ induces a weak geometric action on $D^b(X_{R'})$ by using a base change of the kernels. (That's what we saw in the case of the $Br_{aff}$ action on $D^b(coh(T^*G/B))$).

\bs

We will use the following observation, 
\begin{observation}[Relations form]
	Let X be a smooth scheme over R. Let H be a group. Let H be generated by a set S of elements $h\in S\subset H$. The relations can be described in the form 
	\begin{equation}
	\label{relation form}
	h^{\pm}_{i_1}...h^{\pm}_{i_k}=1, h^{\pm}_{i_j}\in S
	\end{equation}
	 To define a weak action of H on the category $D^b(Coh(X))$, let $\F_h:D^b(Coh(X))\rightarrow D^b(Coh(X))$ be an equivalence functor assigned to $h\in S$, which is given by a kernel $\mathcal{F}\in D^b(Coh(X\times X))$. Let $(\F_h)^{-1}:=\F_h^{-1}$ defined to be the kernel of the inverse equivalence. Then these functors define a weak action of H on $D^b(Coh(X))$ if there are isomorphisms of sheaves (in the derived category) corresponding to the relations (\ref{relation form}):
	
	\begin{equation}
	\label{relation form2}
	\F_{h^{\pm}_{i_1}}\star ....\star \F_{h^{\pm}_{i_k}}\simeq O_{\Delta}
	\end{equation}
	Where $O_{\Delta}$
	is the structure sheaf of the diagonal. 
	
\end{observation}

\bs

Next, Let $B_{0}\subset P_{0}$. We need to understand the pullback functor of D modules $D^b(\mathcal{D}_{\lambda}(G/P_{0})_k-mod_0)\rightarrow D^b(\mathcal{D}_{\lambda}(G/B_{0})_k-mod_0)$, at the level of the corresponding bounded derived category of coherent sheaves, \cite{BM}.

\subsection*{The pullback functor $\pi^*:D^b(Coh_0(T^*G/P_{0})_k)\rightarrow D^b(Coh_0(T^*G/B_{0})_k)$}
Working over $k$ of char $p>>0$. Let $\pi:(G/B_0)_k\rightarrow (G/P_0)_k$ be the projection.
The ordinary pullback functor at the level of derived category of D modules, $D^b(\D_{\lambda}(T^*G/P_0)_k-mod_0)\rightarrow D^b(\D_{\lambda}(T^*G/B_0)_k-mod_0)$ does not correspond to the ordinary pullback functor at the level of derived category of coherent sheaves.  Rather the corresponding pullback of coherent sheaves is:

	\subsubsection*{Definition of the special pullback functor $Coh_0 (T^*(G/P_0)_k)\rightarrow Coh_0( T^*(G/B_0)_k)$}
					
					We describe a functor $Coh(T^*G/P_0)_R\rightarrow Coh(T^*G/B_0)_R$(R as before). In the case R=k, and after restricting to $Coh_0(T^*G/P_0)_k$ this is compatible with the pullback functor of D modules.

				\begin{remark}
				To simplify notations, I remove mentioning the base ring R from the notation. 
				\end{remark}

				Let $B\subset G$ be a Borel subgroup. Let $P\supset B$ be a parabolic which contains it.
				Consider the varieties X:=G/B, Y:=G/P.
				Consider the projection $\pi:X\rightarrow Y$, a smooth surjective map. 
				Consider the morphism of vector bundles $d\pi:T^*Y\times_Y X\rightarrow T^*X$.  Since $\pi$ is smooth, d$\pi$ is a closed embedding.
				
					\begin{example}Let G=$GL_n$, Let P be a minimal parabolic, in standard form, its Levi is (n-2) boxes of size one and one box of size two by two. Consider the projection $G/B\rightarrow G/P$. We can describe the varieties and this projection in terms of moduli space of flags. 
			
					Identifying G/B with the moduli space of full flags of length n and G/P with the moduli space of partial flags, ${0\subset V_1\subset V_2\subset...V_{i-1}\subset V_{i+1}\subset...V_n}$ with $dim V_i=i$, the morphism $G/B\rightarrow G/P$ is forgetting $V_i$. The fiber is $\mathbb{P}^1$ (It corresponds to choices of a line in $V_{i+1}/V_{i-1}$). The closed embedding df is a divisor.

		\end{example}

		\begin{example}[Of df]
					Consider the case P=G. 
					
					X=G/B, Y=G/G=pt $\Rightarrow T^*Y=pt$ 
					$d\pi:X\rightarrow T^*X$ is the zero section.

			\end{example}

					Let $f: T^*(G/P)\times_{G/P}G/B\rightarrow T^*(G/P)$	
				be the projection map. $f$ is smooth and flat. 
					
				\begin{defi}\cite{BM}			
				Using the correspondence $d\pi,f$ define the special pullback functor $d\pi_*f^*:D^bCoh(T^*(G/P))\rightarrow D^bCoh(T^*(G/B))$
				
				\end{defi}
				
					Since f is flat and $d\pi$ is a closed embedding, it follows that, 
				\begin{claim} This functor is exact.
				\end{claim}

				Consider the case where $R=k$ (algebraically closed field of char p $>h$).
				Restricting the domain to $D^bCoh_0(T^*(G/P_0))$, consider the functor $d\pi_*f^*:D^bCoh_0(T^*(G/P_0))\rightarrow D^bCoh_0(T^*(G/B_0))$. Under the equivalence of Coherent with D modules, the functor $d\pi_*f^*:D^bCoh_0(T^*(G/P_0))\rightarrow D^bCoh_0(T^*(G/B_0))$ corresponds to the pullback of D modules $D^b(\mathcal{D}_{\lambda}(G/P_0)-mod_0)\rightarrow D^b(\mathcal{D}_{\lambda}(G/B_0)-mod_0)$. Under the equivalence with the categories of representations, it corresponds to the functor $D^b(A_{\lambda}^0-mod)\rightarrow D^b(u_{\lambda}^0-mod)$ that is induced by the surjective morphism of algebras $u_{\lambda}^0\rightarrow A_{\lambda}^0$

		 \begin{remark}
		
		As mentioned above the affine braid group action on $D^b(Coh(T^*G/B_0))$ has another version when considering the category $D^b(Coh(\tilde{\g}))$, where $\tilde{\g}$ is the Grothendieck resolution. We can construct a local system for the Parbaolic version, $D^b(Coh(\tilde{\g}_P))$. (At the level of k points $\g^P:=((x,\mathfrak{p})| \mathfrak{p}\in G/P, x\in \mathfrak{p}$)). 
		
		 	In this version the situation is simpler. The functor between coherent sheaves that corresponds to the pullback functor of D modules is the ordinary pullback functor of coherent sheaves 
		\end{remark}
				
		\begin{remark}
		We don't discuss the case	$X=\tilde{\g}_P$ in length since it's not a symplectic resolution, just a Poisson variety. (Studying the family of quantization of a Poisson variety, may allow to generalize the picture that we suggest of a local system on a subset of the universal parameter space of quantizations to such case.)
		\end{remark}

\subsection*{Extending to coherent sheaves with no support condition}

In this subsection, R=k.

We will define a local system on $V^0_{\CC}$ with the value $D^b(Coh(T^*G/P_0))$, s.t restriction of the functors to the category $D^bCoh_0(T^*G/P_0)$ recovers the previous local system. 

Remember, the generators of the groupoid of $V^0_{\CC}$ are $l_{\lambda,\mu}$, where $\lambda,\mu$ are two weights in adjucent real alcoves that share a codimension one wall. 

In the local system for $D^bCoh_0(T^*G/P_0)$ we attached the functor $F_{\lambda,\mu,Coh_{P_0}}$ to $l_{\lambda,\mu}$.

Let $\widetilde{F_{\lambda,\mu,Coh_{P_0}}}$ denote the functor we will attach to the generator $l_{\lambda,\mu}$ in the local system for $D^bCoh(T^*G/P_0)$

\bs

When $P_0$ is a Borel $B_0$, we know a definition of $\widetilde{F_{\lambda,\mu,Coh_{B_0}}}$ by its kernel, whose restriction to the subcategory with the support condition is $F_{\lambda,\mu,Coh_{B_0}}$. The kernels live in the heart. In the abelian category $Coh(T^*G/B_0 \times T^*G/B_0)$

\bs

Let $\pi^*$ be the special pullback functor we defined above.

\begin{observation}
\label{diagram support}
The following diagram commutes. 

\[
\xymatrix
{D^b(Coh_0(T^*G/B_0))\ar[r]^-{F_{\lambda,\mu,Coh_{B_0}}} &{D^b(Coh_0(T^*G/B_0))}\\
D^b(Coh_0(T^*G/P_0))\ar[r]_{F_{\lambda,\mu,Coh_{P_0}}} \ar[u]^-{\pi^*}& {D^b(Coh_0(T^*G/P_0))} \ar[u]_{\pi^*}}
\]
	
Moreover, the functors $F_{\lambda,\mu,Coh_{P_0}}$ are uniquely characterized as functors that makes the diagram commutes. 
\end{observation}

\begin{claim}
\label{claim}
There exists and unique functors $\widetilde{F_{\lambda,\mu,Coh_{P_0}}}$ that complete the following diagram to a commutative diagram. 

\[
\xymatrix
{D^b(Coh(T^*G/B_0))\ar[r]^-{\widetilde{F_{\lambda,\mu,Coh_{B_0}}}} &{D^b(Coh(T^*G/B_0))}\\
D^b(Coh(T^*G/P_0))\ar@{-->}[r]_{\widetilde{F_{\lambda,\mu,Coh_{P_0}}}} \ar[u]^-{\pi^*}& {D^b(Coh(T^*G/P_0))} \ar[u]_{\pi^*}}
\]
\end{claim}

\begin{prop}
\label{prop}
The functors $\widetilde{F_{\lambda,\mu,Coh_{P_0}}}$ defined above, define a local system on $V^0_{\CC}$ with values the category $D^bCoh(T^*G/P_0)$
\end{prop}

\begin{proof}
We need to check that the relations of the groupoid hold. These relations are described in terms of equations on the kernels of the form, equation (\ref{relation form2}). Here we use a lemma 

\begin{lemma}
Let X be a smooth projective variety over k that admits a projective map to an affine variety.
	A relation of the form of equation (\ref{relation form2}) for sheaves in $D^b(Coh(T^*X\times T^*X))$, holds iff it holds after passing to a formal neighborhood of the zero section $X\subset T^*X$. (\cite{BM})
\end{lemma}

Using the lemma, it's enough to prove equation (\ref{relation form2}) for the restriction of the kernels to the formal neighborhood of the zero section. This restriction of the kernels corresponds to restricting the functors to the subcategory $D^bCoh_0(T^*X)$.
Restriction of the functors $\widetilde{F}_{\lambda,\mu,Coh_{P_0}}$are the functors $F_{\lambda,\mu,Coh_{P_0}}$. (Since we know this holds for $\widetilde{F_{\lambda,\mu,Coh_{B_0}}}$ and by the uniqueness part in observation (\ref{diagram support})). The restricted functors do form a local system, hence equation (\ref{relation form2}) is satisfied for their kernels. 

\end{proof}

\subsection*{Extending to case of char 0}

Let G be over a $R:=\Z[1/h]$. Let $B_0\subset P_0$ be the fixed borel and fixed parabolic. We will define a local system on $V_{\CC}^0$ with value the category $D^bCoh(T^*G/P_0)$, s.t base changing the functors to the category $D^b(Coh(T^*G/P_0))_k$, recovers the previous local system. 

Remember the generators of the groupoid are $l_{\lambda,\mu}$ as before. Let $\overline{{F}_{\lambda,\mu,Coh_{P_0}}}$ denote the functors we will attach to $l_{\lambda,\mu}$ in the local system for the category $D^bCoh(T^*G/P_0)$.

Consider the case $P_0=B_0$ a Borel. 
In this case we already explained the existence of the functors $\overline{F}_{\lambda,\mu,Coh_{B_0}}$, defined by their kernel, that form a local system, and whose base extension to k, recovers the the local system on $V^0_{\CC}$ with value $D^bCoh(T^*G/B_0)_k$.

\begin{claim}
\label{claim2}
There exists and unique functors $\overline{F}_{\lambda,\mu,Coh_{P_0}}:D^b(Coh(T^*G/P_0))\rightarrow D^b(Coh(T^*G/P_0))$ that make the following diagram commute.
\[
\xymatrix
{D^b(Coh(T^*G/B_0))\ar[r]^-{\overline{F_{\lambda,\mu,Coh_{B_0}}}} &{D^b(Coh(T^*G/B_0))}\\
D^b(Coh(T^*G/P_0))\ar@{-->}[r]_{\overline{F_{\lambda,\mu,Coh_{P_0}}}} \ar[u]^-{\pi^*}& {D^b(Coh(T^*G/P_0))} \ar[u]_{\pi^*}}
\]

\end{claim}

\bs

\begin{prop}
Attaching these functors to the generators $l_{\lambda,\mu}$ defines a local system on $V^0_{\CC}$ with the value $D^b(CohT^*G/P_0)$, whose base change to k, is $\widetilde{F_{\lambda,\mu,Coh_{P_0}}}$
\end{prop}

\begin{proof}
For the second statement in the proposition, look at the base extension to k of the diagram in claim \ref{claim2}. Since $\overline{F_{\lambda,\mu,CohB_0}}_k\simeq \widetilde{F_{\lambda,\mu,CohB_0}}$ and by the uniqueness of claim \ref{claim}, it follows that $\overline{F_{\lambda,\mu,CohP_0}}_k\simeq \widetilde{F_{\lambda,\mu,CohP_0}}$.

\bs

For the first statement, to prove the functors $\overline{F_{\lambda,\mu,Coh_{P_0}}}$ satisfy the relations of a local system, we need to prove that the equations of the form , equation (\ref{relation form2}), on the convolution of the kernels of $\overline{F_{\lambda,\mu,Coh P_0}}$ hold. For this we need a lemma.

\begin{lemma}
 Let R be a finite localization of $\Z$. Let X=$T^*G/B$  Where X  is considered as a variety over R. Let  $\Delta_{X}\subset X\times_R X$ denote the diagonal and let $\Delta_{X,\bar{F_q}}$ denote the base extension of this diagonal to $\bar{F_q}$. q a prime number. 

Let $\mathcal{F}\in D^bCoh^{\mathbb{G}_m}(X\times_R X)$.
If for any prime q, which is not invertible in R, the base extension of $\mathcal{F}$ to $\overline{F_p}$  (in the derived sense), satisfies the relation $\mathcal{F}\otimes_R \overline{F_q}\simeq O_{\Delta_{X_{\overline{F_q}}}}$
Then $\mathcal{F}\simeq O_{\Delta_{X_R}}$

\end{lemma}

Let A be a finitely generated flat R algebra. Let $A_q:=A\otimes_R \overline{F_q}$. The lemma above follows from the next claim about A modules and their derived extension to $\overline{F_q}$ for q which is prime in $\Z$, which is not invertible in R.(see \cite{BR})

\begin{claim} Let M be an object in the bounded derived category of A modules. If the derived extension of of $M$ to $\overline{F_q}$ is concentrated in degree 0 for every prime q not invertible in R, then M is flat over R and concentrated in degree 0. If the derived extension is zero for every q as above, then M is zero.
\end{claim}

Using the lemma, we reduce to proving the relations (\ref{relation form2}) holds after base extension to $\bar{F_q}$ for various primes q, not invertible in R. After base extension this becomes relations for the kernels of $\widetilde{F_{\lambda,\mu,Coh_{P_0}}}_{\bar{F_q}}$. By proposition \ref{prop} these functors form a local system hence their kernels satisfy the relations.

\end{proof}

\begin{remark}[An application of the local system]
Let X be a symplectic resolution over k.
Recall that $V_{\R}$ is a parameter space for quantization of X. These quantizations give rise to natural t structures on the category of coherent sheaves $D^b(Coh(X))$. In particular, in the area where localization theorem holds $V^0_{\R}$, there is a t structure attached to each alcove. It's possible to describe the changes in the t structures when crossing a wall, using the local system above, and a variant of Lusztig a-function for the Weyl group. We discuss that in \cite{B2} for the case $X=T^*G/P$. 
\end{remark}

\end{document}